\documentclass[a4paper,12pt]{article}

\usepackage[textwidth=125mm, textheight=195mm]{geometry}
\usepackage[english]{babel}
\usepackage{graphicx}
\usepackage{amsmath}
\usepackage{amsfonts}
\usepackage{amsthm}
\usepackage{epstopdf}
\usepackage{color}

\begin{document}

\newcommand{\wk}{\mbox{$\,<$\hspace{-5pt}\footnotesize )$\,$}}

\numberwithin{equation}{section}
\newtheorem{teo}{Theorem}
\newtheorem{lemma}{Lemma}

\newtheorem{coro}{Corollary}
\newtheorem{prop}{Proposition}
\theoremstyle{definition}
\newtheorem{definition}{Definition}
\theoremstyle{remark}
\newtheorem{remark}{Remark}

\newtheorem{scho}{Scholium}
\newtheorem{open}{Question}
\newtheorem{example}{Example}
\numberwithin{example}{section}
\numberwithin{lemma}{section}
\numberwithin{prop}{section}
\numberwithin{teo}{section}
\numberwithin{definition}{section}
\numberwithin{coro}{section}
\numberwithin{figure}{section}
\numberwithin{remark}{section}
\numberwithin{scho}{section}

\bibliographystyle{abbrv}

\title{Asymmetry measures for convex distance functions}
\date{}

\author{Vitor Balestro  \\ Instituto de Matem\'{a}tica e Estat\'{i}stica \\ Universidade Federal Fluminense \\ 24210201 Niter\'{o}i \\ Brazil \\ vitorbalestro@id.uff.br \and Horst Martini \\ Fakult\"{a}t f\"{u}r Mathematik \\ Technische Universit\"{a}t Chemnitz \\ 09107 Chemnitz\\ Germany \\ martini@mathematik.tu-chemnitz.de \and Ralph Teixeira \\ Instituto de Matem\'{a}tica e Estat\'{i}stica \\ Universidade Federal Fluminense \\ 24210201 Niter\'{o}i \\ Brazil \\ ralph@mat.uff.br}

\maketitle

\begin{abstract} \emph{Gauges}, or \emph{convex distance functions} are, roughly speaking, norms without symmetry. In this paper we intend to quantify how asymmetric a planar gauge can be. We introduce asymmetry measures for smooth gauges and for strictly convex gauges, prove that they are invariant under isometries, and investigate lower and upper bounds for them. Identifying a gauge with a convex body containing the origin in its interior (the \emph{unit ball} of the gauge), we also prove that all introduced asymmetry measures are continuous in the Hausdorff distance. Finally, we show that, modifying one of the constructed asymmetry measures, a certain duality principle holds. 
\end{abstract}



\section{Introduction}\label{intro}

A \emph{gauge} (or \emph{convex distance function}) in a linear vector space is, roughly speaking, the non-symmetric distance obtained from a convex body containing the origin as an interior point via the well-known Minkowski functional. The topic of gauge spaces has been vigorously studied in the past few years (see, e.g., \cite{BBJM}, \cite{jahn}, \cite{Jahn1}, \cite{Jahn2}, \cite{Mer-Jahn-Rich}, and \cite{obst}), and the objective of this paper is to relate this theory to the theory of asymmetry measures for convex bodies (see, e.g., the recent papers \cite{brandenberg}, \cite{Bra-Mer}, and \cite{Bra-Mer2}, and also the survey \cite{grunbaum} and the monograph \cite{toth}). We introduce new asymmetry constants \emph{which make sense from the metric viewpoint}, since they are invariant under gauge isometries. 

Let us start by explaining formally the objects that we will study. A \emph{gauge} in a vector space $X$ is a map $\gamma:X\rightarrow\mathbb{R}$ with the following properties:\\

\noindent\textbf{i. positivity and nondegeneracy:} $\gamma(x) \geq 0$ for any $x \in X$, and equality holds if and only if $x = 0$,\\

\noindent\textbf{ii. positive homogeneity:} $\gamma(\alpha x) = \alpha\gamma(x)$ for any $x \in X$ and $\alpha \geq 0$,\\

\noindent\textbf{iii. triangle inequality:} $\gamma(x+y) \leq \gamma(x) + \gamma(y)$ for all $x,y \in X$.\\

A vector space endowed with a gauge is called a \emph{gauge space}. The set $B_{\gamma} := \{x \in X:\gamma(x) \leq 1\}$ is called the \emph{unit ball} of the gauge space $(X,\gamma)$, and the boundary $\partial B_{\gamma} := \{x\in X:\gamma(x) = 1\}$ of $B_{\gamma}$ is the \emph{unit sphere} of $(X,\gamma)$. If $X$ is a \emph{plane} (that is, a two-dimensional vector space), then the unit ball and unit sphere are called \emph{unit disk} and \emph{unit circle}, respectively. For simplicity, we will use the letter $K$ to denote the unit ball (or disk) of a gauge space.

It is easy to see that $B_{\gamma}$ is a \emph{convex body} (a compact, convex set with non-empty interior) which contains the origin $o_X$ in its interior. Conversely, if $K \subseteq X$ is a convex body such that $o_X \in \mathrm{int}K$, then $K$ yields a gauge function defined by the \emph{Minkowski functional}
\begin{align*} \gamma_K(x) := \inf\{\alpha \geq 0:x \in \alpha K\}.
\end{align*}
It follows that the choice of a gauge in a vector space is equivalent to the choice of a convex body and one of its interior points (as the origin). Notice in particular that if the convex body is symmetric with respect to the chosen interior point, then the gauge is a \emph{norm} (and the converse also holds, of course). 

This discussion shows that the theory of gauges is closely related to convex geometry. For that reason, one can naturally use measures of asymmetry of convex bodies also to study geometric properties of gauge spaces. In general, we want to discuss the following question, mainly inspired by the works by Gr\"{u}nbaum \cite{grunbaum} and T\'{o}th \cite{toth}.\\

\noindent\textbf{Main Question.} Define and investigate continuous (in some suitable topology) functionals on (certain) spaces of convex bodies with the following properties:\\

\noindent\textbf{i.} their minimum values are attained precisely in the class of centrally symmetric convex bodies,\\

\noindent\textbf{ii.} they make sense from the viewpoint of the metric endowed by the chosen convex body (in particular, we want them to be invariant under \emph{isometries} of the corresponding gauge space), \\

\noindent\textbf{iii.} they respect some kind of ``duality principle" (with respect to the \emph{dual gauge}, which we will define later). \\

Let us formalize the notions that we need to tackle this question. A gauge $\gamma$ in a vector space $X$ induces a not necessarily symmetric distance $d_{\gamma}$ by
\begin{align*} d_{\gamma}(x,y) = \gamma(y-x),
\end{align*}
for any $x,y \in X$. A \emph{gauge space isometry} (or simply \emph{isometry}) is a surjective map $F:(X,\gamma_X)\rightarrow (Y,\gamma_Y)$ such that 
\begin{align*} d_X(x,y) = d_Y(Fx,Fy),
\end{align*}
for any $x,y \in X$, where $d_X$ and $d_Y$ denote the distances induced by $\gamma_X$ and $\gamma_Y$, respectively. Notice that we are demanding an isometry to be surjective merely because we are more interested in this case. As in the case of norms, isometries between gauge spaces fixing the origin are always linear. This is stated below, and for a proof we refer the reader to \cite{Ba-Ma-Tei1}.
\begin{teo}[Mazur-Ulam theorem for gauges] If $F:(X,\gamma_X)\rightarrow (Y,\gamma_Y)$ is an isometry such that $F(o_X) = o_Y$, then $F$ is linear.
\end{teo}

If there exists an isometry between $(X,\gamma_X)$ and $(Y,\gamma_Y)$, then we say that these gauge spaces are \emph{isometric}. From the Mazur-Ulam theorem it follows immediately that any isometry is a  surjective \emph{affine transformation}, that is, a composition of a linear map with a translation. Consequently, if two gauge spaces are isometric, then they are \emph{linearly isometric}, meaning that they are isometric with respect to a linear transformation. 

\begin{coro} Two gauge spaces $(X,\gamma_X)$ and $(Y,\gamma_Y)$ are isometric if and only if there exists an isomorphism mapping of the unit ball $K_X$ of $(X,\gamma_X)$ onto the unit ball $K_Y$ of $(Y,\gamma_Y)$.
\end{coro}
\begin{proof} See \cite[Corollary 2.2]{Ba-Ma-Tei1}. 

\end{proof}

\begin{remark} The corollary above has a very important consequence for our theory. The asymmetry measures that we will define and study are invariant under gauge space isometries, and hence, if we regard them as asymmetry measures of their respective unit balls, then we have asymmetry measures for convex bodies which are invariant under affine transformations. This is another ``desirable" feature of an asymmetry measure for convex bodies.\\

\end{remark}

A gauge space is said to be \emph{smooth} if its unit ball $K$ is a \emph{smooth convex body}, which means that it is supported by exactly one hyperplane at each boundary point. If the boundary $\partial K$ of $K$ does not contain a line segment, then $K$ and the associated gauge space are called \emph{strictly convex}. For simplicity, througout the text we use the following notation:\\

\noindent$\bullet$ $\mathcal{K}_{\mathrm{o}}(X)$ is the family of convex bodies containing the origin as an interior point,\\

\noindent$\bullet$ $\mathcal{K}_{\mathrm{o}}^{\mathrm{sm}}(X)$ is the subset of the above consisting of all smooth convex bodies,\\

\noindent$\bullet$ $\mathcal{K}_{\mathrm{o}}^{\mathrm{sc}}(X)$ is the set of strictly convex bodies containing the origin as an interior point. 

\section{Dual gauges and orthogonality}

Recall that the \emph{dual space} $X^*$ of a vector space $X$ is the set of all linear functionals $f:X\rightarrow\mathbb{R}$. For a given body $K\subseteq \mathcal{K}_{\mathrm{o}}(X)$, we define the \emph{polar body} $K^{\circ}\subseteq X^*$ as
\begin{align*} K^{\circ} = \{f\in X^*:f(x) \leq 1 \ \mathrm{for \ all} \ x \in K\}.
\end{align*}
It is a known fact that if $K \in \mathcal{K}_{\mathrm{o}}^{\mathrm{sm}}(X)$, then $K^{\circ} \in \mathcal{K}_{\mathrm{o}}^{\mathrm{sc}}(X^*)$, and also if $K \in \mathcal{K}^{\mathrm{sc}}_{\mathrm{o}}(X)$, then $K^{\circ} \in \mathcal{K}_{\mathrm{o}}^{\mathrm{sm}}(X^*)$. For more in this direction we refer the reader to Chapter 1 of Schneider's book \cite{schneider}. Of course, the polar body induces a gauge $\gamma^*$ in $X^*$ by the Minkowski functional, and we call it the \emph{polar gauge}. 

If $X$ is even-dimensional, then the fixation of a symplectic form $\omega$ on $X$ yields an isomorphism $\mathcal{I}:X^*\rightarrow X$ by contraction in the first coordinate:
\begin{align*} f(\cdot) = \iota_{\mathcal{I}(f)}(\cdot) = \omega(\mathcal{I}(f),\cdot).
\end{align*}
The image of $K^{\circ}$ under the identification $\mathcal{I}$ will be denoted by $K^{\omega}$, and will be called the \emph{dual body}. Since obviously $K^{\omega} \in \mathcal{K}_{\mathrm{o}}(X)$, we have that it induces a gauge in $X$, which we call the \emph{dual gauge}. If $\gamma$ is the gauge whose unit ball is $K$, then we denote the dual gauge by $\gamma_{\omega}$. Notice that it is clear that $\mathcal{I}$ is an isometry between the gauge spaces $(X^*,\gamma^*)$ and $(X,\gamma_{\omega})$. 

We proceed now to introduce an orthogonality relation in a gauge space (for more on orthogonality in gauge spaces we refer the reader to \cite{jahn}). If $(X,\gamma)$ is a gauge space whose unit ball is $K$, then we say that a vector $x \in X$ is \emph{orthogonal} to a vector $y \in X$ whenever
\begin{align*} \gamma(x) \leq \gamma(x+ty),
\end{align*}
for any $t \in \mathbb{R}$. We denote this relation by $x \dashv y$ and mention that it is the natural analogue of \emph{Birkhoff orthogonality} for normed spaces (see \cite{alonso}). 

Geometrically, if $x,y \in X$ are non-zero vectors, then $x \dashv y$ if and only if the line in the direction of $y$ supports $K$ at $x/\gamma(x)$. As a consequence, if the dimension of $X$ is greater than $2$, we have that a hyperplane $H$ supports $K$ at a point $x \in \partial K$ if and only if $x\dashv y$ for any $y \in H$. It follows that the orthogonality relation can be extended to a relation between vectors and hyperplanes. \\

Notice that the orthogonality also has the following properties:\\

\noindent\textbf{i. right-homogeneity:} If $x\dashv y$, then $x\dashv ty$ for any $t \in \mathbb{R}$,\\

\noindent\textbf{ii. positive left-homogeneity:} If $x\dashv y$, then $tx \dashv y$ for any $t > 0$. \\

For proofs of these properties, and also of the geometric interpretation described above, we refer the reader to \cite[Section 4]{Ba-Ma-Tei1}. Moreover, notice that, up to re-scaling, the orthogonality relation has the uniqueness property on the right if and only if the unit disk $K$ of the gauge is smooth, and uniqueness on the left holds if and only if $K$ is strictly convex.

Next, we enunciate the results which describe the relation between orthogonality and duality in the two-dimensional case. This will be of vital importance for us.

\begin{teo}\label{orthodual} Let $(X,\gamma)$ be a two-dimensional gauge space endowed with a symplectic form $\omega$, and let $\gamma_{\omega}$ be the corresponding dual gauge. Then we have the inequality
\begin{align*} \omega(x,y) \leq \gamma(x)\gamma_{\omega}(y),
\end{align*}
with equality if and only if $x\dashv y$ and $\omega(x,y) \geq 0$. Consequently, if $x \dashv y$ and $\omega(y,x) > 0$, then $y \dashv_{\omega} x$, where $\dashv_{\omega}$ denotes the orthogonality relation of the dual gauge $\gamma_{\omega}$. 
\end{teo}
\begin{proof} See Proposition 5.2 and Corollary 5.1 in \cite{Ba-Ma-Tei1}. 

\end{proof}

\section{Continuity}\label{continuity}

Let $\mathcal{K}(X)$ denote the space of convex bodies in $X$. Assume that $X$ is endowed with any norm topology, and denote by $d_X$ and $B$ the induced distance and unit ball of this norm, respectively. It is a well known fact that this space is a (complete) metric space with the distance
\begin{align*} d_H(K,L) = \min\{\lambda\geq0: K\subseteq L +\lambda B \ \mathrm{and} \ L\subseteq K +\lambda B\},
\end{align*}
for any $K,L \in \mathcal{K}(X)$. Of course, the Hausdorff distance does not depend on the norm fixed on $X$ (neither does the topology of $X$). The next proposition is a useful characterization of convergence in the Hausdorff metric. It is obtained by combining Theorem 1.8.8 and Lemma 1.8.1 of \cite{schneider}.

\begin{prop}\label{convbound} We have that $K_n\rightarrow K$ in $\mathcal{K}(X)$ if and only if the following holds:\\

\noindent\textbf{\emph{i.}} For each $x \in\partial K$ there exists a sequence $(x_n)$ such that $x_n \in \partial K_n$ for each $n\in\mathbb{N}$, and $x_n \rightarrow x$ (in the metric $d_X$),\\

\noindent\textbf{\emph{ii.}} the limit of any convergent sequence $(x_n)$ with $x_n \in \partial K_n$ is a point of $\partial K$. 
\end{prop}

The space of gauges in a given vector space $X$ is identified with the subset $\mathcal{K}_{\mathrm{o}}(X)$ of convex bodies which contain the origin $o_X$ as an interior point. Therefore, the Hausdorff distance induced to $\mathcal{K}_{\mathrm{o}}(X)$ defines a distance between two gauges in a vector space $X$. Our objective in this section is to investigate whether the orthogonality relation is continuous in the Hausdorff metric. First we need two technical lemmas. The first one states that gauges given by ``close" convex bodies are also ``close". The other one says that the convergence of gauges implies the convergence of the respective dual gauges.

\begin{lemma}\label{gaugeestimate} Let $K\in \mathcal{K}_{\mathrm{o}}(X)$ be a fixed convex body. There is a number $\varepsilon(K)>0$ such that for any $0 < \varepsilon <\varepsilon(K)$ there exists a number $\delta = \delta(K,\varepsilon)$ with the property that if $L \in \mathcal{K}_{\mathrm{o}}(X)$ is such that $d_H(K,L) < \varepsilon$, then
\begin{align*} \frac{\gamma_K(x)}{1+\delta}\leq \gamma_L(x) \leq (1+\delta)\gamma_K(x),
\end{align*}
for any $x \in X$, where $\gamma_K$ and $\gamma_L$ are the gauges whose unit balls are $K$ and $L$, respectively. Moreover, we can choose $\delta$ in such a way that $\delta \rightarrow 0$ as $\varepsilon \rightarrow 0$. 
\end{lemma}
\begin{proof} For a fixed convex body $K$ there exist numbers $\alpha,\beta > 0$ such that $B\subseteq \alpha K$ and $K \subseteq \beta B$.  Hence, if $d_H(K,L) < \varepsilon$, we have
\begin{align*} L\subseteq K +\varepsilon B \subseteq K+(\alpha\varepsilon) K \ \ \ \mathrm{and} \ \ \ K \subseteq L + \varepsilon B \subseteq L+(\alpha\varepsilon)K.
\end{align*}
Making $\varepsilon(K) = 1/\alpha$ and taking $0 < \varepsilon < 1/\alpha$, we claim that $(1-\alpha\varepsilon)K\subseteq L$. Indeed, applying the second inclusion above inductively, we get
\begin{align*} K\subseteq (1+\alpha\varepsilon + \ldots + (\alpha\varepsilon)^n)L+(\alpha\varepsilon)^{n+1}K.
\end{align*}
Since $0 <\alpha\varepsilon < 1$, letting $n\rightarrow+\infty$ we get
\begin{align*} K\subseteq \left(1+\sum_{j=0}^{\infty}(\alpha\varepsilon)^j\right)L = \frac{1}{1-\alpha\varepsilon}L,
\end{align*}
from where our claim follows. If $x \in X$ is a non-zero vector, then
\begin{align*} \frac{(1-\alpha\varepsilon)x}{\gamma_K(x)} \in (1-\alpha\varepsilon)K\subseteq L,
\end{align*}
which means that this vector has length less than or equal to $1$ in the gauge $\gamma_L$. It follows immediately that
\begin{align*} \gamma_L(x) \leq \frac{1}{1-\alpha\varepsilon}\gamma_K(x) = \left(1+\frac{\alpha\varepsilon}{1-\alpha\varepsilon}\right)\gamma_K(x).
\end{align*}
On the other hand, since $L\subseteq (1+\alpha\varepsilon)K$ we have that
\begin{align*} \frac{\gamma_K(x)}{1+\alpha\varepsilon} \leq \gamma_L(x),
\end{align*}
for any $x \in X$. Finally, setting
\begin{align*} \delta = \frac{\alpha\varepsilon}{1-\alpha\varepsilon},
\end{align*}
we have $\delta > \alpha\varepsilon$, and hence 
\begin{align*} \frac{\gamma_K(x)}{1+\delta} \leq \frac{\gamma_K(x)}{1+\alpha\varepsilon} \leq \gamma_L(x) \leq \left(1+\frac{\alpha\varepsilon}{1-\alpha\varepsilon}\right)\gamma_K(x) = (1+\delta)\gamma_K(x),
\end{align*}
as we wished. Notice that $\alpha$ only depends on $K$, from where $\delta$ only depends on $K$ and $\varepsilon$.

\end{proof}

In what follows, \emph{convergence in $\mathcal{K}_\mathrm{o}(X)$} (or in $\mathcal{K}_{\mathrm{o}}^{\mathrm{sm}}(X)$ or in $\mathcal{K}_{\mathrm{o}}^{\mathrm{sc}}(X)$) will always mean convergence in the sense of the Hausdorff distance. Also, recall that the dual gauge of a smooth gauge is strictly convex, and vice-versa (see, e.g., \cite{Ba-Ma-Tei1}).

\begin{lemma}\label{convdual} If $K_n\rightarrow K$ in $\mathcal{K}_{\mathrm{o}}(X)$, then $K^{\omega}_n\rightarrow K^{\omega}$ in $\mathcal{K}_{\mathrm{o}}(X)$.
\end{lemma}
\begin{proof} First of all, notice that if $K,L\in\mathcal{K}_{\mathrm{o}}(X)$ are such that $K\subseteq L$, then $L^{\omega}\subseteq K^{\omega}$. Indeed, since both bodies contain the origin as an interior point we have, for each $x \in X$, that
\begin{align*} \max\{\omega(x,y):y \in K\} \leq \max\{\omega(x,y): y\in L\}.
\end{align*}
Hence, if $x \in L^{\omega}$, than the latter is less than or equal to $1$, and consequently the same holds for the left hand side. It follows that $x \in K^{\omega}$. Notice that, similarly, one can prove that if $K\in\mathcal{K}_{\mathrm{o}}(X)$ and $\lambda > 0$, then
\begin{align*} (\lambda K)^{\omega}  = \frac{1}{\lambda}K^{\omega}.
\end{align*}
As in the previous lemma, let $\alpha > 0$ be such that $B\subseteq \alpha K$. If $0 < \varepsilon <1$, then the covergence $K_n\rightarrow K$ guarantees that there exists $n_0 \in \mathbb{N}$ such that $n> n_0$ implies
\begin{align*} K\subseteq K_n + \frac{\varepsilon}{\alpha}B \subseteq K_n + \varepsilon K \ \ \ \mathrm{and} \ \ \ K_n \subseteq K + \frac{\varepsilon}{\alpha}B \subseteq K+\varepsilon K. 
\end{align*}
From the second inclusion above, we get that
\begin{align*} \frac{1}{1+\varepsilon}K^{\omega} \subseteq K_n^{\omega},
\end{align*}
which means that
\begin{align}\label{convdualinc1} K^{\omega} \subseteq K^{\omega}_n + \varepsilon K^{\omega}_n.
\end{align}
On the other hand, applying the first inclusion inductively as in the previous lemma, we get, for any $k\in\mathbb{N}$, that
\begin{align*} K\subseteq (1+\varepsilon+\varepsilon^2+\ldots+\varepsilon^k)K_n + \varepsilon^{k+1}K_n,
\end{align*}
and this gives
\begin{align*} K\subseteq \frac{1}{1-\varepsilon}K_n.
\end{align*}
The inclusion $(1-\varepsilon)K^{\omega}_n \subseteq K^{\omega}$ follows immediately, and this can be written as
\begin{align*} K^{\omega}_n \subseteq \frac{1}{1-\varepsilon}K^{\omega} = K^{\omega}+\frac{\varepsilon}{1-\varepsilon}K^{\omega}.
\end{align*}
Let $\beta > 0$ be such that $K^{\omega}\subseteq \beta B$. Then from the above we have
\begin{align}\label{convdualinc2} K_n^{\omega} \subseteq K^{\omega} + \frac{\beta\varepsilon}{1-\varepsilon}B \left(\subseteq \beta B + \frac{\beta \varepsilon}{1-\varepsilon}B\right).
\end{align}
Using the inclusion between the parentheses above in (\ref{convdualinc1}) we get
\begin{align}\label{convdualinc3} K^{\omega} \subseteq K_n^{\omega} + \left(\varepsilon\beta + \frac{\beta\varepsilon^2}{1-\varepsilon}\right)B.
\end{align}
A straightforward calculation shows that if we replace $\varepsilon$ by $\varepsilon/(\beta+\varepsilon)$, then the numbers multiplying $B$ in (\ref{convdualinc2}) and (\ref{convdualinc3}) become (both) equal to $\varepsilon$. Hence, if we choose $n_1 \in \mathbb{N}$ such that $n>n_1$ implies $d_H(K,K_n) < \varepsilon/\alpha(\beta + \varepsilon)$, then we have from (\ref{convdualinc2}) and (\ref{convdualinc3}) that
\begin{align*} K_n^{\omega} \subseteq K^{\omega} + \varepsilon B \ \ \ \mathrm{and} \ \ \ K^{\omega} \subseteq K^{\omega}_n + \varepsilon B,
\end{align*}
that is, $d_H(K^{\omega},K_n^{\omega}) < \varepsilon$. This shows that $K^{\omega}_n \rightarrow K^{\omega}$. 

\end{proof}

\begin{remark} Since $(K^{\omega})^{\omega} = -K$, we also have that $K_n^{\omega}\rightarrow K^{\omega}$ implies $K_n\rightarrow K$. \\
\end{remark}

The next theorems state that, in some sense, the orthogonality relation for gauge spaces is continuous in the Hausdorff metric. 

\begin{teo}[Right continuity]\label{cont1} Assume that $K_n\rightarrow K$ in $\mathcal{K}_{\mathrm{o}}^{\mathrm{sm}}(X)$, and let $(x_n)$ be a sequence in $X$ such that $x_n \in \partial K_n$ for each $n\in\mathbb{N}$, and $x_n$ converges (in the metric $d_X$) for some point $x \in \partial K$. Suppose that $(y_n)$ is the (unique) sequence in $X$ such that
\begin{align*} x_n \dashv_n y_n, \ \ \ and \ \ \ \omega(x_n,y_n) = 1
\end{align*}
for each $n\in\mathbb{N}$, where $\dashv_n$ denotes the orthogonality relation of the gauge whose unit ball is $K_n$. Let also $y \in X$ be (the unique vector) such that 
\begin{align*} x\dashv y, \  \ \ and \ \ \ \omega(x,y) = 1,
\end{align*}
where $\dashv$ denotes the orthogonality relation in the gauge given by $K$ as unit ball. Then $y_n$ converges to $y$ in the metric $d_X$. 
\end{teo}
\begin{proof} For each $n\in\mathbb{N}$, denote by $\gamma^n$ and $\gamma^n_{\omega}$ the gauge whose unit ball is $K_n$ and its dual gauge, respectively. Also, denote by $\gamma$ and $\gamma_{\omega}$ the gauge given by $K$ as unit ball and its dual gauge, respectively. From Lemma \ref{gaugeestimate} and Lemma \ref{convdual} it follows that there exists a sequence $\delta_n$ of positive numbers such that $\delta_n \rightarrow 0$, and
\begin{align*} \frac{\gamma_{\omega}(x)}{1+\delta_n} \leq \gamma^n_{\omega}(x) \leq (1+\delta_n)\gamma_{\omega}(x)
\end{align*}
for any $n \in \mathbb{N}$. From the hypothesis and Theorem \ref{orthodual} we have that $\gamma^n_{\omega}(y_n) = 1$ for each $n\in\mathbb{N}$, and hence we have in particular that $\gamma_{\omega}(y_n) \leq 1+\delta_n$. It follows that $(y_n)$ is a bounded sequence in the gauge given by the convex body $K^{\omega}$. Hence, $(y_n)$ is also bounded in the metric $d_X$ (indeed, we have $\alpha B \subseteq K^{\omega}\subseteq\beta B$ for some positive numbers $\alpha$ and $\beta$). It follows that $(y_n)$ has a convergent subsequence $(y_{n_k})$, with $y_{n_k} \rightarrow y_0$, say. Setting $x = y_{n_k}$ in the inequality above, we get
\begin{align*} \frac{\gamma_{\omega}(y_{n_k})}{1+\delta_{n_k}} \leq \gamma_{\omega}^{n_k}(y_{n_k})  \leq (1+\delta_{n_k})\gamma_{\omega}(y_{n_k}),
\end{align*}
and letting $k\rightarrow\infty$ (and remembering that $\gamma_{\omega}^{n_k}(y_{n_k}) = 1$ for every $k$) yields $\gamma_{\omega}(y_0) = 1$. From the continuity of the symplectic form, it follows that
\begin{align*} \omega(x,y_0) = \lim_{k\rightarrow\infty}\omega(x_{n_k},y_{n_k}) = 1 = \gamma_{\omega}(y_0)\gamma(x),
\end{align*}
and this gives that $x \dashv y_0$. Thus, $y_0 = y$. Moreover, this argument also shows that \textbf{any} convergent subsequence of $(y_n)$ converges to $y$. Therefore, $(y_n)$ is a bounded sequence in $X$ with the property that every convergent subsequence has the same limit. From standard analysis arguments it follows that $y_n \rightarrow y$.

\end{proof}

\begin{teo}[Left continuity]\label{cont2} Suppose that $K_n\rightarrow K$ in $\mathcal{K}^{\mathrm{sc}}_{\mathrm{o}}(X)$, and let $(x_n)$ be a sequence in $X$ such that $x_n \in \partial K_n$ for each $n \in\mathbb{N}$. Assume also that $x_n$ converges in the metric $d_X$ to a point $x \in \partial K$. Let $(y_n)$ be the sequence in $X$ such that
\begin{align*} y_n \dashv_n x_n \ \ \ and \ \ \ \omega(y_n,x_n) = 1,
\end{align*}
and let $y \in X$ be such that $y \dashv x$ and $\omega(y,x) = 1$. Then $y_n \rightarrow y$ in the metric $d_X$. 
\end{teo}
\begin{proof} From Lemma \ref{convdual}, we have that $K_n^{\omega}\rightarrow K^{\omega}$. Define the sequences
\begin{align*} \hat{x}_n = \frac{-x_n}{\gamma_{\omega}^n(-x_n)} \ \ \mathrm{and} \\
\hat{y}_n = \frac{y_n}{\gamma^n(-y_n)},
\end{align*}
for each $n \in \mathbb{N}$. Denoting by $\dashv^n_{\omega}$ the orthogonality relation in the dual gauge whose unit ball is $K_n^{\omega}$, we have from Theorem \ref{orthodual} that $-x_n \dashv_{\omega}^n y_n$, from where we get $\hat{x}_n \dashv^n_{\omega} \hat{y}_n$. Moreover, recall that the bi-dual gauge reverses the sign of the unit ball, that is, $(\gamma_{\omega})_{\omega} = \gamma_{-K}$, and recall also that we always have $\gamma_{-K}(x) = \gamma_K(-x)$. Hence, applying Theorem \ref{orthodual} for the orthogonality relation $\dashv_{\omega}^n$ yields
\begin{align}\label{eq1leftcont} 1 = \omega(-x_n,y_n) = \gamma_{\omega}^n(-x_n)\gamma^n(-y_n)
\end{align}
for each $n \in \mathbb{N}$, and we get that
\begin{align*} \omega(\hat{x}_n,\hat{y}_n) = \frac{\omega(-x_n,y_n)}{\gamma_{\omega}^n(-x_n)\gamma^n(-y_n)} = 1,
\end{align*}
for every $n \in \mathbb{N}$. Moreover, it is clear that the sequence $(\hat{x}_n)$ is such that $\hat{x}_n \in \partial K_n^{\omega}$ for each $n\in\mathbb{N}$, and from Lemma \ref{gaugeestimate} we have that $\gamma^n_{\omega}(-x_n) \rightarrow \gamma_{\omega}(-x)$ as $n\rightarrow\infty$. It follows that
\begin{align*} \hat{x}_n\rightarrow\hat{x} := \frac{-x}{\gamma_{\omega}(-x)} \in \partial K^{\omega}
\end{align*}
when $n\rightarrow\infty$. Hence, it comes straightforwardly from Theorem \ref{cont1} (applied to the convergence in the dual gauges) that $\hat{y}_n$ converges to a vector $\hat{y}$ with the properties that $\hat{x} \dashv_{\omega} \hat{y}$ and $\omega(\hat{x},\hat{y}) = 1$. Also, letting $n\rightarrow \infty$ in equality (\ref{eq1leftcont}) we get
\begin{align*} \lim_{n\rightarrow\infty}\gamma^n(-y_n) = \frac{1}{\gamma_{\omega}(-x)},
\end{align*}
from where it follows that 
\begin{align*} y_n \rightarrow y := \frac{\hat{y}}{\gamma_{\omega}(-x)}.
\end{align*}
Now, notice that
\begin{align*} 1 = \omega(\hat{x},\hat{y}) = \frac{\omega(-x,y)\gamma_{\omega}(-x)}{\gamma_{\omega}(-x)} = \omega(-x,y) = \omega(y,x).
\end{align*}
Finally, positive rescaling of $\hat{x}\dashv_{\omega}\hat{y}$ gives $-x\dashv_{\omega}y$, which yields that $-x \dashv_{\omega} -y$. From this, together with the equality $\omega(-y,-x) = 1$, we get that $-y \dashv_{-K} -x$. Here we are using again the fact that the bi-dual gauge reverses the sign of the unit ball, and we are denoting by $\dashv_{-K}$ the orthogonality relation of the gauge whose unit ball is $(K^{\omega})^{\omega} = -K$. Finally, the relation $-y \dashv_{-K} -x$ is clearly equivalent to $y \dashv x$. This concludes the proof.

\end{proof}

\section{The smooth case}\label{outer}

In this section we explore a way to measure ``how far from being a normed space" a gauge space  
is, or (in other words) ``how asymmetric with respect to the origin" a convex body is. This question is not new, and a lot of asymmetry measures (also called symmetry measures) were already introduced and studied. We refer the reader to the survey \cite{grunbaum} and to the book \cite{toth}.

Let $(X,\gamma)$ be a smooth gauge plane whose unit disk is $K$. Also, assume that $\omega$ is a fixed symplectic form on $X$. We start with the smooth case because of the following uniqueness results.

\begin{lemma} For any non-zero vector $x \in X$ there exists precisely one vector $b^+\!(x) \in \partial K$ such that $x \dashv b^+\!(x)$ and $\omega(x,b^+\!(x)) > 0$. 
\end{lemma}
\begin{proof} Since orthogonality is positively left-homogeneous, it suffices to consider that $x \in \partial K$. Due to smoothness, there exists a unique supporting line to $K$ at $x$. Hence there are precisely two unit vectors which are right-orthogonal to $x$. Of course, they have the same direction but opposite orientations. It follows that only one of them is ``correctly" oriented.

\end{proof}

\begin{coro} Given a non-zero vector $x \in X$, there exists a unique vector $b^-\!(x) \in \partial K$ such that $x \dashv b^-\!(x)$ and $\omega(x,b^-\!(x)) < 0$. 
\end{coro}

With this in mind, we construct a map $b:\partial K\rightarrow \partial K$ which associates each $x \in \partial K$ to the vector 
\begin{align*} b(x) = b^+\!(x) - b^-\!(x).
\end{align*}
Notice that we may write $b^-\!(x) = \alpha b^+\!(x)$ for some number $\alpha < 0$, and hence $\gamma(b^+\!(x) - b^-\!(x)) = (1-\alpha)\gamma(b^+\!(x)) > 0$, from which it follows that $b(x)$ is always a non-zero vector. Of course, $b$ can be naturally extended (if necessary) to every non-zero vector because $b^+$ and $b^-$ are positively homogeneous. Hence we simply put $b(\lambda x) = b(x)$ for any $x \in \partial K$ and $\lambda > 0$. 
\begin{lemma} The map $b^+(x):\partial K\rightarrow \partial K$ is continuous, and it is a bijection if and only if $K$ is strictly convex. 
\end{lemma}
\begin{proof} The continuity is straightforward, and we will omit the proof. To verify that $b$ is injective if $K$ is strictly convex, just notice that even if $K$ has parallel supporting lines at $x,z \in \partial K$, we will have that $b^+(x)$ and $b^+(z)$ have opposite orientations, and hence they are distinct. This argument also gives that $b^+$ is onto, because any direction $v$ of $X$ supports $K$ at two points, and each one of them will be associated to each one of the intersections of the direction $v$ with $\partial K$. On the other hand it is clear that if $\partial K$ contains a line segment, then $b^+$ is not injective.

\end{proof}

For each $x \in \partial K$, let $p(x)$ be the intersection of $\partial K$ with the ray starting at the origin $o_X$ in the direction $-x$. We can look at this as a map $p:\partial K\rightarrow\partial K$, with the property that
\begin{align*} p(p(x)) = x,
\end{align*}
for any $x\in\partial K$. Figure \ref{asymmetry} illustrates these constructions. Notice that both maps $p$ and $b$ depend on the origin $o_X$ of the vector space, and not only on the unit disk $K$. 
\begin{figure}[h]
\centering
\includegraphics{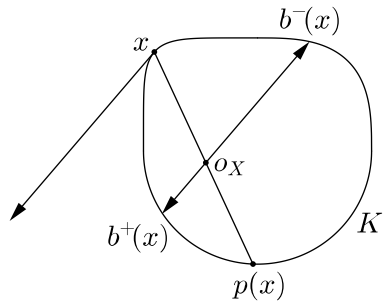}
\caption{The maps $b$ and $p$.}
\label{asymmetry}
\end{figure}

Finally, with all this in mind, we define the \emph{outer asymmetry function} $f_{\mathrm{out}}:\partial K\rightarrow\mathbb{R}$ of $(X,\gamma)$ as
\begin{align*} f_{\mathrm{out}}(x) = \frac{\omega(b(x),b(p(x))}{\lambda_{\omega}(K)},
\end{align*}
where $\lambda_{\omega}(K)$ is the area of the unit disk $K$ given by $\omega$ (notice that we divide by this area, and so the asymmetry function is invariant, up to the sign, under rescaling the symplectic form). Of course, if the gauge $\gamma$ is a norm, then $f_{\mathrm{as}} = 0,$ because in this case the supporting lines of symmetric points are always parallel (and the converse is also true, as we will see in the next proposition). This is the reason why $f_{\mathrm{out}}$ measures somehow the asymmetry of the gauge with respect to the origin.

\begin{teo}\label{norm1} If $f_{\mathrm{out}} = 0$, then $\gamma$ is a norm. 
\end{teo}
\begin{proof} A segment joining two points of the boundary $\partial K$ at which $K$ has parallel supporting lines is called an \emph{affine diameter} of $K$. Our proof is a direct consequence of the following result: if $x$ is an interior point of $K$ such that every chord through $x$ is an affine diameter, then $K$ is symmetric with respect to $x$. Regarding a proof of that, we refer the reader to \cite{busemann} and \cite{hammer} (and for more on affine diameters of convex bodies, see the survey \cite{soltan}). The condition $f_{\mathrm{out}} = 0$ gives that $b(x)$ and $b(p(x))$ are parallel for every $x \in \partial K$ (recall that $b$ does not vanish), which means that all chords through the origin are affine diameters. It follows that the unit disk $K$ of $\gamma$ is symmetric with respect to the origin, and hence $\gamma$ is a norm.  

\end{proof}

The outer asymmetry function is also the tool to answer (in the smooth case) a natural question regarding orthogonality: in any gauge plane $(X,\gamma)$, can we guarantee the existence of a non-zero vector $x \in X$ for which there exists a non-zero vector $y \in X$ such that $x \dashv y$ and $-x \dashv y$? We will investigate the smooth case. 
\begin{prop}\label{orth1} Let $(X,\gamma)$ be a smooth gauge plane. Then there exist non-zero vectors $x,y \in X$ such that $x \dashv y$ and $-x \dashv y$. 
\end{prop}
\begin{proof} This comes immediately from the continuity of the outer asymmetry function. For any $z \in \partial K$, we have that $f_{\mathrm{out}}(z) = -f_{\mathrm{out}}(p(z))$. Hence there exists a vector $x \in \partial K$ such that $f_{\mathrm{out}}(x) = 0$, which means that $b(x)$ and $b(p(x))$ are in the same direction $y$, say. It follows that $x \dashv y$ and $p(x) \dashv y$. Due to positive left-homogeneity (which means that if $v$ is orthogonal do $w$, then any positive multiple of $v$ is also orthogonal to $w$), we get from the latter that $-x \dashv y$. 

\end{proof}

\begin{remark} Notice that this is the same as saying that every interior point of a smooth convex body $K$ lies in some affine diameter, or that $K$ equals the union of its affine diameters. 
\end{remark}

The next step is to understand what happens to the outer asymmetry function under a linear isometry. First we need a technical lemma regarding the behavior of symplectic forms on planes under isomorphisms.

\begin{lemma}\label{lemmasymp} Let $T:(X,\omega_X)\rightarrow (Y,\omega_Y)$ be an isomorphism between two-dimensional symplectic vector spaces. Then there exists a number $\alpha \neq 0$ such that
\begin{align*} \omega_Y(Tx,Tz) = \alpha\omega_X(x,z),
\end{align*}
for any $x,z \in X$.
\end{lemma}
\begin{proof} Let $x_0,z_0 \in X$ be linearly independent vectors such that $\omega_X(x_0,z_0) = 1$. Then we set
\begin{align*} \alpha = \omega_Y(Tx_0,Tz_0).
\end{align*}
The equality $\omega_Y(Tx,Tz) = \alpha\omega_X(x,z)$ comes immediately from linearity. 

\end{proof}

\begin{remark} Of course, if both $X$ and $Y$ represent $\mathbb{R}^2$, and if both $\omega_X$ and $\omega_Y$ represent the usual determinant (seen as a symplectic form), then $\alpha = \mathrm{det}T$.
\end{remark}

From this lemma, it is easy to see that either $\omega_Y(Tx,Tz)$ and $\omega_X(x,z)$ have always the same sign ($\alpha > 0$) or always opposite signs ($\alpha < 0$) for any linearly independent $x,z \in X$. In the first case, we say that $T$ is \emph{orientation preserving}, while in the second case we say that $T$ is \emph{orientation reversing}. 

\begin{prop}\label{comm1} Let $(X,\gamma_X,\omega_X)$ and $(Y,\gamma_Y,\omega_Y)$ be smooth gauge planes, and let $T:X\rightarrow Y$ be an orientation preserving isometry. Denote by $b_X, p_X, b_Y, p_Y$ the maps $b$ and $p$ defined as above for the gauge planes $X$ and $Y$, respectively. Hence we have 
\begin{align*} b_Y(Tx) = T(b_X(x)) \ and \\ p_Y(Tx) = T(p_X(x)),
\end{align*}
for any $x \in X$. 
\end{prop}
\begin{proof} Let $x \in X$. Since $T$ is a linear isometry, we have immediately from the definition that
\begin{align*} \gamma_Y(Tx+tT(b^+_X(x))) = \gamma_X(x + tb^+_X(x)) \geq \gamma_X(x) = \gamma_Y(Tx),
\end{align*}
for any $t \in \mathbb{R}$. It follows that $Tx \dashv_Y T(b^+_X(x))$. Again using the fact that $T$ is a linear isometry, we get
\begin{align*} \gamma_Y(T(b^+_X(x)) = \gamma_X(b^+_X(x)) = 1.
\end{align*}
Finally, since $T$ is orientation preserving, we have that $\omega_Y(Tx,T(b^+_X(x))) > 0$, from which we obtain
\begin{align*} b^+_Y(Tx) = T(b^+_X(x)).
\end{align*}
With the same argument, we have that $b^-_Y(Tx) = T(b^-_X(x))$. Thus, due to linearity we get
\begin{align*} b_Y(Tx) = b^+_Y(Tx) - b^-_Y(Tx) = T(b^+_X(x)) - T(b^-_X(x)) = T(b_X(x)).
\end{align*}
For the other claim, notice that $T(p_X(x))$ is a vector of $\partial K_Y$ which is clearly a negative multiple of $Tx$, because $p_X(x)$ is a negative multiple of $x$. Hence $T(p_X(x)) = p_Y(Tx)$. 

\end{proof}

Of course, the second equality still holds if $T$ is orientation reversing. This is not the case for the first equality, though. Indeed, in this case we have $\omega_Y(Tx,T(b_X(x))) < 0$.

\begin{coro}\label{comm2} If $T$ is an orientation reversing linear isometry, then we still have $p_Y(Tx) = T(p_X(x))$ for any $x \in \partial K$, but the other equality becomes
\begin{align*} b_Y(Tx) = - T(b_X(x)),
\end{align*}
for every $x \in \partial K$. 
\end{coro}
\begin{proof} For the equality $p_Y(Tx) = T(p_X(x))$ the proof is exactly the same as for the orientation preserving case. For the other equality, notice that $T(b^+_X(x))$ is a vector such that $\gamma_Y(T(b^+_X(x))) = 1$ and $Tx \dashv_Y T(b^+_X(x))$, but $\omega_Y(Tx,T(b^+_X(x))) < 0$. Hence we have
\begin{align*} b^-_Y(Tx) = T(b^+_X(x)),
\end{align*}
and similarly we get that $b^+_Y(Tx) = T(b^-_X(x))$. Therefore, the desired equality comes from linearity, as in the orientation preserving case. 

\end{proof}

The function $f_{\mathrm{out}}$ explains the asymmetry of a gauge, since it calculates how non-parallel the supporting lines of opposite points (with respect to the origin) of the unit circle $\partial K$ are. Therefore, it is natural to quantify this asymmetry by taking a maximum. We introduce the \emph{constant of outer asymmetry} of a smooth gauge plane $(X,\gamma)$ by
\begin{align*} c_{\mathrm{out}}(\gamma) = \max\{|f_{\mathrm{out}}(x)|:x \in \partial K\}.
\end{align*}
Of course, this constant does not depend on the symplectic form fixed on $X$, thus being purely metric in nature. Taking this into consideration, one might expect that the constant of outer asymmetry is invariant under isometries. This is indeed true, as we will prove next. 

\begin{teo}\label{outisometry} Let $T:(X,\gamma_X,\omega_X)\rightarrow(Y,\gamma_Y,\omega_Y)$ be a linear isometry between gauge planes. Then $T$ preserves the outer asymmetry function up to the sign, which is reversed if and only if $T$ is orientation reversing. In particular, if $(X,\gamma_X)$ and $(Y,\gamma_Y)$ are isometric gauge planes, then $c_{\mathrm{out}}(\gamma_X) = c_{\mathrm{out}}(\gamma_Y)$. 
\end{teo}
\begin{proof} Denote by $f_X$ and $f_Y$ the outer asymmetry functions of $X$ and $Y$, respectively, and recall that from Lemma \ref{lemmasymp} we can write
\begin{align*} \omega_Y(Tx,Tz) = \alpha\omega_X(x,z),
\end{align*}
for any $x,z \in X$, where $\alpha$ is a non-zero constant. Denoting, again, by $K_X$ and $K_Y$ the unit disks of $(X,\gamma_X)$ and $(Y,\gamma_Y)$, respectively, we notice that
\begin{align*} \lambda_{\omega_Y}(K_Y) = \lambda_{\omega_Y}(T(K_X)) = |\alpha|\lambda_{\omega_X}(K_X).
\end{align*}
Hence, if $T$ is orientation preserving, then $\alpha > 0$, and we get from Proposition \ref{comm1} that
\begin{align*} f_Y(Tx) = \frac{\omega_Y(b_Y(Tx),b_Y(p_Y(Tx)))}{\lambda_{\omega_Y}(K_Y)} = \frac{\omega_Y(T(b_X(x)),T(b_X(p_X(x))))}{\alpha\lambda_{\omega_X}(K_X)} = \\ =  \frac{\alpha\omega_X(b_X(x),p_X(b_X(x)))}{\alpha\lambda_{\omega_X}(K_X)} = f_X(x),
\end{align*}
for any $x \in \partial K$. However, if $T$ is orientation reversing, then $\alpha < 0$, and we get from Corollary \ref{comm2} that
\begin{align*} f_Y(Tx) = \frac{\omega_Y(b_Y(Tx)),b_Y(p_Y(Tx)))}{\lambda_{\omega_Y}(K_Y)} = \frac{\omega_Y(-T(b_X(x)),-T(b_X(p_X(x))))}{-\alpha\lambda_{\omega_X}(K_X)} = \\ = \frac{\alpha\omega_X(b_X(x),b_X(p_X(x)))}{-\alpha\lambda_{\omega_X}(K_X)} = -f_X(x).
\end{align*}
In both cases, we have that $|f_Y(Tx)| = |f_X(x)|$ for any $x \in \partial K$. Since the restriction of $T$ to $\partial K_X$ is a bijection onto $\partial K_Y$, we get that $c_{\mathrm{out}}(\gamma_X) = c_{\mathrm{out}}(\gamma_Y)$. 

\end{proof}

We already know that $0$ is a lower bound for $c_{\mathrm{out}}$, and that the equality holds if and only if the gauge is a norm. Next we will provide an upper bound for the constant of outer asymmetry which is not, however, attained by any gauge plane. 

\begin{teo} For any smooth gauge $\gamma$ in a two-dimensional vector space $X$ we have that $c_{\mathrm{out}}(\gamma) < 2$. Moreover, this upper bound is sharp.
\end{teo}
\begin{proof} One can readily see that, for each $x \in \partial K$, the number $|\omega(b(x),b(p(x))|$ equals twice the area of the quadrilateral (inscribed in $\partial K$) whose vertices are $b^+(x)$, $b^-(x)$, $b^+(p(x))$, and $b^-(p(x))$. Hence, it follows from the smoothness of the unit disk $K$ that
\begin{align*} |f_{\mathrm{out}}(x)| = \frac{|\omega(b(x),b(p(x))|}{\lambda_{\omega}(K)} < 2,
\end{align*}
for each $x \in \partial K$. It follows that $c_{\mathrm{out}}(\gamma) \leq 2$. To verify that the bound is strict indeed, notice that the map $\partial K \ni x \mapsto |f_{\mathrm{out}}(x)|$ is continuous, and defined over a compact set. Hence it reaches a maximum value for some $x_0 \in \partial K$, and thus we have $c_{\mathrm{out}}(\gamma) = |f_{\mathrm{out}}(x_0)| < 2$. 

Now, given any $\varepsilon >0$, we construct a unit ball $K$ and choose $x \in \partial K$ such that $f_{\mathrm{{out}}}(x) >2-\varepsilon $. Let $X = \mathbb{R}^2$ and assume that $\omega$ is the usual determinant. Given $\alpha >0$, consider the hexagon of vertices $A(0,1) $, $B( -1,0) $, $C( -1,-1) $, $D(0,-1) $, $E(\alpha ,0) $ and $F\left(\frac{\alpha }{
1+\alpha },\frac{\alpha }{1+\alpha }\right)$. (The vertex $F$ was chosen to be
on the segment $AE$ such that $AB$ and $CF$ are parallel, see Figure \ref{sharp}.) One can compute $\lambda_{\omega}(BCEF) =2\alpha +1$ while $%
\lambda_{\omega}(ABCDEF) =2\alpha +2$, so, by taking $\alpha $ big enough,
the ratio between these areas can be taken as close to $1$ as wished (say, $%
1-\varepsilon /4$). Now, one can slightly smooth the polygon at the corners making sure that the following conditions are satisfied:\\

\noindent\textbf{i.} The vertices $A$, $B$, $C$, $D$ and $E$ are kept unchanged;\\

\noindent\textbf{ii.} the vertex $F$ shifts such that $CF$ is kept parallel to $AB$;\\

\noindent\textbf{iii.} the tangent line at $D$ is horizontal;\\

\noindent\textbf{iv.} the tangent line at $A$ has the direction of $AB$.\\

\noindent Let $K$ be this ``smoothed polygon", and take $x=D\in \partial K$. Then $p(x) =A$, $b(
x) = E-B$ and $b(p(x))= C-F$. So
\begin{align*}
f_{\mathrm{out}}(x) =\frac{\omega(b(x) ,b(p(x)))}{\lambda _{\omega}(K)},
\end{align*}
and since the smoothing can be done with as little area alteration as
wanted, this will be arbitrarily close to%
\begin{align*}
\frac{2\lambda_{\omega}(BCEF)}{\lambda_{\omega}(ABCDEF)}=2-\frac{%
\varepsilon }{2}\text{.}
\end{align*}

\end{proof}

\begin{figure}[h]
\centering
\includegraphics{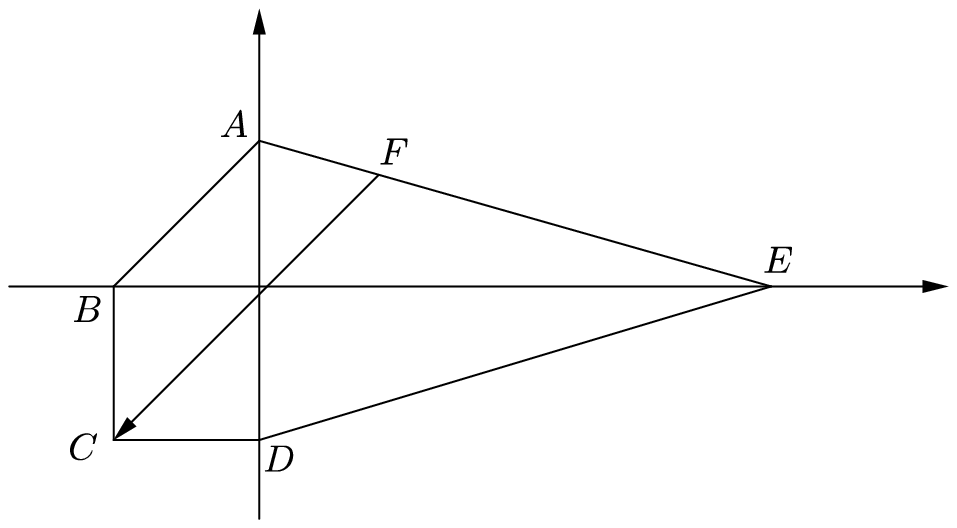}
\caption{The upper bound for $c_{\mathrm{out}}$ is sharp.}
\label{sharp}
\end{figure}

Recall that the restriction of the Hausdorff distance to $\mathcal{K}_{\mathrm{o}}^{\mathrm{sm}}(X)$ can be interpreted as a distance between gauge spaces. We will use the results of Section \ref{continuity} to prove that the constant of outer asymmetry is continuous with respect to this distance. In what follows, we will work with sequences of gauges, and therefore we covenant that $\gamma^n$ denotes the gauge whose unit disk is $K_n$. Also, $\gamma$ still stands for the gauge given by $K$ as unit disk.

\begin{teo}\label{hausdorffouter} The constant of outer asymmetry is continuous with respect to the Hausdorff distance in the sense that $c_{\mathrm{out}}(\gamma^n)\rightarrow c_{\mathrm{out}}(\gamma)$ if $K_n\rightarrow K$ in $\mathcal{K}_{\mathrm{o}}^{\mathrm{sm}}(X)$.
\end{teo}
\begin{proof} First we prove that if a subsequence of $c_{\mathrm{out}}(\gamma^n)$ converges, then it has to converge to $c_{\mathrm{out}}(\gamma)$. Let $x \in \partial K$. For simplicity, we will abuse of the notation and denote such a subsequence by $c_{\mathrm{out}}(\gamma^n)$ still. From Proposition \ref{convbound}, we can take a sequence $x_n \in \partial K_n$ such that $x_n \rightarrow x$. From Theorem \ref{cont1} we have that $b_n(x_n) \rightarrow b(x)$ in any norm metric of $X$, where $b_n$ denotes the previously defined map $b$ for the gauge plane whose unit disk is $K_n$. Since the map $p$ and the volume functional are continuous in the Hausdorff distance (for the volume, see \cite[Theorem 1.8.20]{schneider}), and since $\omega$ is continuous in the norm topology of $X$, we have that
\begin{align*}| f^n_{\mathrm{out}}(x_n)| \rightarrow |f_{\mathrm{out}}(x)|.
\end{align*}
If $x \in \partial K$ is a point where $c_{\mathrm{out}}(\gamma)$ is attained, then we have from the above that
\begin{align*}\lim_{n\rightarrow\infty}c_{\mathrm{out}}(\gamma^n) \geq c_{\mathrm{out}}(\gamma).
\end{align*}
Now notice that if the inequality is strict, then we can find a number $\varepsilon > 0$ and a sequence $(z_n)$ of points such that $z_n \in \partial K_n$ for each $n\in\mathbb{N}$, and
\begin{align*} |f^n_{\mathrm{out}}(z_n)| > c_{\mathrm{out}}(\gamma) + \varepsilon,
\end{align*}
for every $n\in\mathbb{N}$. The sequence $(z_n)$ is clearly bounded. Hence, passing to a subsequence if necessary and using Proposition \ref{convbound} again, we may assume that $z_n \rightarrow z \in \partial K$. Thus, the inequality above gives
\begin{align*} c_{\mathrm{out}}(\gamma)+\varepsilon < |f_{\mathrm{out}}^n(z_n)| \rightarrow |f_{\mathrm{out}}(z)| \leq c_{\mathrm{out}}(\gamma),
\end{align*}
which is a contradiction. Finally, it is clear that $c_{\mathrm{out}}(\gamma^n)$ is a bounded sequence of real numbers with the property that any of its convergent subsequences has $c_{\mathrm{out}}(\gamma)$ as limit. It follows that $c_{\mathrm{out}}(\gamma^n)\rightarrow c_{\mathrm{out}}(\gamma)$ as $n\rightarrow\infty$. 

\end{proof}

\section{The strictly convex case}\label{inner}

In the previous section we investigated a (purely metrical and isometrically invariant) way to measure asymmetry of a smooth gauge plane $(X,\gamma)$, and in particular we verified that this can be used to prove that there exist non-zero vectors $x,y \in X$ satisfying $x \dashv y$ and $-x \dashv y$. In this section we aim to do something similar for (possibly non-smooth, but) strictly convex planar gauges. 

Throughout this section, $(X,\gamma)$ stands for a strictly convex gauge plane whose unit disk is denoted, as usual, by $K$. First of all, we need a standard result from convex geometry which explains why our construction only works in the strictly convex case. 

\begin{prop} Let $X$ be a two-dimensional vector space, and let $K \subseteq X$ be a convex body. If $K$ is strictly convex, then $K$ is supported by each direction of $X$ at precisely two points of $\partial K$. 
\end{prop}

We refer to \cite{schneider} for a proof. Taking this into consideration, for each $x \in \partial K$ we let $a^+(x)$ and $a^-(x)$ be the points of $\partial K$ where $K$ is supported by a line in the direction of $x$. These points are chosen such that we have
\begin{align*} \omega(a^+\!(x),x) > 0 \ \mathrm{and} \\ \omega(a^-\!(x),x) < 0,
\end{align*}
where $\omega$ is a given fixed symplectic form on $X$. Observe that this can be done since each one of them lies in the interior of one of the half-planes determined by the line through the origin in the direction of $x$. In what follows, recall that for each $x \in \partial K$ we denote by $p(x)$ the intersection of the unit circle with the ray in the direction of $-x$. 

\begin{prop} The maps $a^+:\partial K\rightarrow\partial K$ and $a^-:\partial K\rightarrow \partial K$ defined as above are continuous. Moreover, we have that $a^+\!\circ p = a^-$ and $a^-\!\circ p = a^+$.
\end{prop}
\begin{proof} The continuity of $a^+$ and $a^-$ follows from the continuity of $\gamma$ if we characterize the support of $K$ in terms of orthogonality, and we will skip the details. The other claims follow from the fact that $\omega(y,x)$ and $\omega(y,p(x))$ have always opposite signs whenever $x$ and $y$ are linearly independent. Figure \ref{innerasymmetry} below illustrates these maps.

\end{proof}

\begin{figure}[h]
\centering
\includegraphics{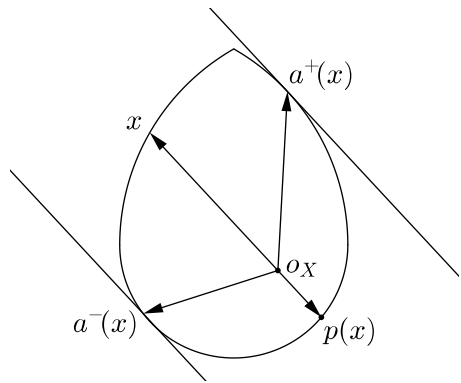}
\caption{The maps $a^+$ and $a^-$.}
\label{innerasymmetry}
\end{figure}

Of course, in the symmetric case we have that $a^+\!(x) = -a^{-}\!(x)$ for any $x \in \partial K$. Inspired by this fact, we define the \emph{inner asymmetry function} $f_{\mathrm{in}}:\partial K\rightarrow\mathbb{R}$ of the gauge plane $(X,\gamma)$ as
\begin{align*} f_{\mathrm{in}}(x) = \frac{\omega(a^+\!(x),a^-\!(x))}{\lambda_{\omega}(K)},
\end{align*}
for each $x \in \partial K$. In the next proposition we investigate some properties of this function.

\begin{prop} The inner asymmetry function is invariant under positive multiplication of the symplectic form, and if $\omega$ is replaced by $-\omega$, then $f_{\mathrm{in}}$ changes (only) its sign. Moreover, if $f_{\mathrm{in}} = 0$, then $\gamma$ is a norm. 
\end{prop}
 \begin{proof} The first claim is obvious, since $\lambda_{\alpha\omega}(K) = |\alpha|\lambda_{\omega}(K)$. Under a negative rescaling, we have that $a^+$ and $a^-$ have exchanged their roles, and this is why $f_{\mathrm{in}}$ changes its sign. For the last claim, notice that if $f_{\mathrm{in}} = 0$, then every affine diameter of $K$ contains the origin $o_X$. Due to the same arguments as used for proving Theorem \ref{norm1}, we have that $K$ is symmetric with respect to the origin, and hence $\gamma$ is a norm. 

\end{proof}

Already at the beginning of this section it was mentioned that one of the motivations to construct this asymmetry constant was to prove, for the non-smooth case, that in any gauge plane there exist opposite directions which are left-orthogonal to the same direction. 

\begin{teo} Let $(X,\gamma)$ be a strictly convex gauge plane. Then there exist non-zero vectors $x,y \in X$ such that $x \dashv y$ and $-x \dashv y$. 
\end{teo}
\begin{proof} Simply notice that $f_{\mathrm{out}}(z) = -f_{\mathrm{out}}(p(z))$ for any $z \in \partial K$. Hence, from the Intermediate Value Theorem, we get that there exists $y \in \partial K$ such that $f_{\mathrm{out}}(y) = 0$. This gives that $a^+\!(y)$ and $a^-\!(y)$ are vectors in the same direction. Setting $x = a^+\!(y)$, we have the desired.

\end{proof}

We want to understand how the inner asymmetry function behaves under an isometry. What happens here is very similar to the case of the outer asymmetry function. It is easy to see that if $T:(X,\gamma_X)\rightarrow (Y,\gamma_Y)$ is a linear isometry which is orientation preserving, then we have
\begin{align*} a_Y^{\pm}(Tx) = T(a^{\pm}_X(x)),
\end{align*}
for any $x \in \partial K$, where $a^{\pm}_X$ and $a^{\pm}_Y$ denote the maps $a^{\pm}$ for the gauge planes $X$ and $Y$, respectively. If $T$ is orientation reversing, then we have
\begin{align*} a_Y^{\pm}(Tx) = T(a^{\mp}_X(x)),
\end{align*}
for any $x \in \partial K$. This leads us to define the \emph{constant of inner asymmetry} as
\begin{align*} c_{\mathrm{in}}(\gamma) = \max\{|f_{\mathrm{in}}(x)|:x \in \partial K\}.
\end{align*}
As in the outer case, this constant is purely metric, meaning that it only depends on the gauge $\gamma$. As expected and as we will see next, it is an isometry invariant. Abusing a little of the notation, we will denote by $f_E$ the inner asymmetry function of a given gauge space $(E,\gamma_E)$.

\begin{teo} Let $T:(X,\gamma_X,\omega_X)\rightarrow (Y,\gamma_Y,\omega_Y)$ be a linear isometry. Then
\begin{align*} f_Y(Tx) = \pm f_X(x),
\end{align*}
for any $x \in \partial K$, where the sign depends on whether $T$ is orientation preserving (positive) or orientation reversing (negative). Consequently, $c_{\mathrm{in}}(\gamma_X) = c_{\mathrm{in}}(\gamma_Y)$. 
\end{teo}
\begin{proof} This comes immediately from the same argument as it was used in Theorem \ref{outisometry}, based on the behavior of the maps $a^{\pm}_X$ and $a^{\pm}_Y$ under $T$ described above. 

\end{proof}

\begin{teo}\label{hausdorffinner} The constant of inner asymmetry is continuous in the Hausdorff distance, in the sense that $c_{\mathrm{in}}(\gamma_n) \rightarrow c_{\mathrm{in}}(\gamma)$ if $K_n \rightarrow K$ in $\mathcal{K}_{\mathrm{o}}^{\mathrm{sc}}(X)$. 
\end{teo}
\begin{proof} From Theorem \ref{cont2}, we have that the maps $a^+$ and $a^-$ are continuous in the Hausdorff metric in the following sense: if $(x_n)$ is a sequence such that $x_n \in \partial K_n$ for each $n \in\mathbb{N}$, with $x_n \rightarrow x \in \partial K$, then $a^+_n(x_n) \rightarrow a^+(x)$ and $a^-_n(x_n)\rightarrow a^-(x)$. Hence we just have to repeat the arguments of Theorem \ref{hausdorffouter}.

\end{proof}

\section{Duality of asymmetry measures}

Our goal in this section is to modify the outer asymmetry function suitably in order to get an asymmetry measure which is dual to the inner asymmetry measure. We do this in such a way that the outer measure of a (smooth) convex body equals the inner measure of its dual body (which is strictly convex). 

Throughout this section, $(X,\gamma)$ denotes a gauge plane whose unit disk $K$ is smooth. As usual, we fix a symplectic form $\omega$ on $X$ and denote by $\gamma_{\omega}$ the dual gauge, and by $K^{\omega}$ its unit disk. We modify the map $b$ defined in Section \ref{outer} to define a new map $\hat{b}:\partial K\rightarrow \partial K^{\omega}$ as
\begin{align*} \hat{b}(x) = \frac{b(x)}{\gamma_{\omega}(x)},
\end{align*}
for each $x \in \partial K$. It is clear that this is a normalization of $b$ in the dual gauge. With this map, we define the \emph{normalized outer asymmetry function} $\hat{f}_{\mathrm{out}}:\partial K\rightarrow \mathbb{R}$ as
\begin{align*} \hat{f}_{\mathrm{out}}(x) = \frac{\omega\Big(\hat{b}(x),\hat{b}(p(x))\Big)}{\lambda_{\omega}(K^{\omega})},
\end{align*}
where the reader may notice that now the normalization is obtained via dividing by the area of the dual disk, instead of that of the disk itself. First we prove that $\hat{f}_{\mathrm{out}}$ is (up to the sign) invariant when changing the symplectic form.
\begin{prop}\label{invprop} The normalized outer asymmetry function remains the same if we replace the symplectic form preserving orientation. If the new symplectic form does not preserve orientation, then $\hat{f}_{\mathrm{out}}$ only changes its sign.
\end{prop}
\begin{proof} Let $\bar{\omega} = \alpha\omega$ for some positive number $\alpha \in \mathbb{R}$. From \cite[Proposition 5.3]{Ba-Ma-Tei1} we get that $K^{\omega} = \alpha K^{\bar{\omega}}$. Hence
\begin{align*} \lambda_{\bar{\omega}}(K^{\bar{\omega}}) = \alpha\lambda_{\omega}(\alpha^{-1}K^{\omega}) = \frac{1}{\alpha}\lambda_{\omega}(K^{\omega}).
\end{align*}
On the other hand, we get that $\gamma_{\bar{\omega}} = \alpha\gamma_{\omega}$. Therefore, it is clear that $\hat{b}_{\bar{\omega}} = \alpha^{-1}\hat{b}_{\omega}$. Thus our first claim comes straightforwardly by calculating $\hat{f}_{\mathrm{out}}$ with respect to $\bar{\omega}$. If we modify the orientation of the symplectic form, notice that we must change the sign of $\hat{b}$. Hence the sign of $\hat{f}_{\mathrm{out}}$ also changes.

\end{proof}

From Proposition \ref{invprop} we have that the normalized outer asymmetry function depends, up to the sign, only on the metric. Next we prove that it is also invariant under an isometry. But before this, we need to understand what happens to the dual gauge under isometries of the original gauge.

\begin{lemma} Let $T:X\rightarrow Y$ be a linear isometry between the gauge planes $(X,\gamma_X,\omega_X)$ and $(Y,\gamma_Y,\omega_Y)$, and let $\alpha \in \mathbb{R}$ be the number such that $\omega_Y(Tx,Ty) = \alpha\omega_X(x,y)$ for any $x,y \in X$. Then $\alpha^{-1}\cdot T:X\rightarrow Y$ is an isometry between the dual gauges.
\end{lemma}
\begin{proof} Let us denote the dual gauges of $\gamma_X$ and $\gamma_Y$ by $\gamma_{\omega_X}$ and $\gamma_{\omega_Y}$, respectively. If $\alpha >0$, then for any $x \in X$ we have
\begin{align*} \gamma_{\omega_Y}(\alpha^{-1}Tx) = \alpha^{-1}\max\{\omega_Y(Tx,y):y \in K_Y\} = \alpha^{-1}\max\{\alpha\omega_X(x,T^{-1}y):y\in K_Y\} = \\ = \max\{\omega_X(x,z):z \in K_X\} = \gamma_{\omega_X}(x).
\end{align*}
Now, if $\alpha < 0$, then for each $x \in X$ we get
\begin{align*} \gamma_{\omega_Y}(\alpha^{-1}Tx) = -\alpha^{-1}\gamma_{\omega_Y}(T(-x))= -\alpha^{-1}\max\{\omega_Y(T(-x),y):y \in K_Y\} = \\ -\alpha^{-1}\max\{\alpha\omega_X(-x,T^{-1}y):y \in K_Y\} = \max\{\omega_X(x,z):z \in K_X\} = \gamma_{\omega_X}(x),
\end{align*}
and this concludes the proof.

\end{proof}

\begin{teo} Let $(X,\gamma_X,\omega_X)$ and $(Y,\gamma_Y,\omega_Y)$ be gauge planes equipped with symplectic forms, and let $T:X\rightarrow Y$ be an isometry between them. We have that $\hat{f}_Y(Tx) = \pm\hat{f}_X(x)$ for any $x \in \partial K_X$, where the sign is positive if and only if $T$ is orientation preserving.
\end{teo}
\begin{proof} Let us assume that $T$ preserves orientation, which means that there exists a number $\alpha > 0$ such that $\omega_Y(Tx,Ty) = \alpha\omega_X(x,y)$ for any $x,y \in X$. From the previous lemma, the map $\alpha^{-1}T$ is an isometry between the gauges, and hence it is clear that for a given $y \in \partial K_Y$ with $y = Tx$ for some $x \in \partial K_X$ we have
\begin{align*} \hat{b}_Y(y) = \frac{b_Y(y)}{\gamma_{\omega_Y}(y)} = \frac{b_Y(Tx)}{\gamma_{\omega_Y}(Tx)} = \frac{Tb_X(x)}{\alpha\gamma_{\omega_Y}(\alpha^{-1}Tx)} = \frac{Tb_X(x)}{\alpha\gamma_{\omega_X}(x)} = \alpha^{-1}T\hat{b}_X(x),
\end{align*}
where we used the previous lemma and Proposition \ref{comm1}. Using these analogous results, we also get immediately that 
\begin{align*}\hat{b}_Y(p_Y(y)) = \frac{Tb_X(p_X(x))}{\alpha\gamma_{\omega_X}(p_X(x))} = \alpha^{-1}T\hat{b}_X(p_X(x)).
\end{align*}
Now, notice that the previous lemma also implies that $K^{\omega_Y} = \alpha^{-1}T(K^{\omega_X})$, where we are abusing of notation. Denote the unit disks of the dual gauges of $X$ and $Y$ by $K^{\omega_X}$ and $K^{\omega_Y}$, respectively. Hence we get
\begin{align*}\lambda_{\omega_Y}(K^{\omega_Y}) = \lambda_{\omega_Y}(\alpha^{-1}T(K^{\omega_X})) = \frac{1}{\alpha^2}\lambda_{\omega_Y}(T(K^{\omega_X})) = \frac{1}{\alpha}\lambda_{\omega_X}(K^{\omega_X}).
\end{align*}
Having this in mind, we calculate the normalized outer asymmetry function of $Y$ (which we will denote as $\hat{f}_Y)$ for a given $y = Tx$ as
\begin{align*} \hat{f}_Y(y) = \frac{\omega_Y\Big(\hat{b}_Y(y),\hat{b}_Y(p_Y(y))\Big)}{\lambda_{\omega_Y}(K^{\omega_Y})} = \frac{\alpha^{-2}\omega_Y\Big( T\hat{b}_X(x),T\hat{b}_X(p_X(x))\Big)}{\alpha^{-1}\lambda_{\omega_X}(K^{\omega_X})} = \\= \frac{\omega_X\Big(\hat{b}_X(x),\hat{b}(p_X(x))\Big)}{\lambda_{\omega_X}(K^{\omega_X})} = \hat{f}_X(x),
\end{align*}
as we wanted to prove. If $T$ is orientation reversing, then from Corollary \ref{comm2} it follows that
\begin{align*} \hat{b}_Y(Tx) = \frac{b_Y(Tx)}{\gamma_{\omega_Y}(b_Y(Tx))} = \frac{-Tb_X(x)}{\gamma_{\omega_Y}(-Tb_X(x))} = \frac{-Tb_X(x)}{-\alpha\gamma_{\omega_Y}(\alpha^{-1}Tb_X(x))} = \frac{Tb_X(x)}{\alpha\gamma_{\omega_X}(b_X(x))} = \\ = \alpha^{-1}T\hat{b}_X(x),
\end{align*}
and, similarly, we get that $\hat{b}_Y(p_Y(Tx)) = \alpha^{-1}T\hat{b}_X(p_X(x))$. However, since now we have that $\alpha < 0$, we get
\begin{align*} \lambda_{\omega_Y}(K^{\omega_Y}) = -\alpha^{-1}\lambda_{\omega_X}(K^{\omega_X}),
\end{align*}
and hence it follows that $\hat{f}_Y(y) = -\hat{f}_X(x)$. 

\end{proof}
Finally, we define the \emph{normalized constant of outer asymmetry} of the gauge $\gamma$ as
\begin{align*} \hat{c}_{\mathrm{out}}(\gamma) = \sup\{|\hat{f}_{\mathrm{out}}(x)| : x \in \partial K\}.
\end{align*}
And, of course, from the last theorem we have the following

\begin{coro} The normalized constant of outer asymmetry of a (smooth) gauge plane $(X,\gamma)$ does not depend on the symplectic form fixed on $X$. Moreover, it is invariant under isometries of $(X,\gamma)$. 
\end{coro}

All of these results justify that the normalized outer asymmetry constant indeed measures asymmetry of a gauge plane. Also, it is easy to see that immediate analogues of Theorem \ref{norm1} and Proposition \ref{orth1} can be proved for the normalized outer symmetry function. Even if this makes the original outer asymmetry function redundant, we see it as a natural first step towards defining dual asymmetry measures. This duality is stated and proved next. 

\begin{teo}\label{duality} Let $(X,\gamma)$ be a smooth gauge plane endowed with a symplectic form $\omega$. Then its dual gauge $\gamma_{\omega}$ is strictly convex, and
\begin{align*} \hat{c}_{\mathrm{out}}(\gamma) = c_{\mathrm{in}}(\gamma_{\omega}).
\end{align*}
In other words, the normalized constant of outer asymmetry of a smooth gauge equals the inner asymmetry constant of the associated dual gauge.
\end{teo}
\begin{proof} For any $x \in \partial K$, define 
\begin{align*} \eta(x) = \frac{x}{\gamma_{\omega}(x)},
\end{align*}
the normalization of $x$ in the dual gauge. Observe that the map $\eta:\partial K\rightarrow \partial K^{\omega}$ defined this way is a bijection. Denote by $a^+_{\omega}$ and $a^-_{\omega}$ the maps associated to the dual gauge $\gamma_{\omega}$ (as in Section \ref{inner}). From the duality of orthogonality introduced in Section \ref{intro} we have that
\begin{align*} \hat{b}(x) = a^-_{\omega}(\eta(x)) \ \mathrm{and} \\
\hat{b}(p(x)) = a^+_{\omega}(\eta(x)),
\end{align*}
for each $x \in \partial K$. Thus
\begin{align*} \hat{f}_{\mathrm{out}}(x) = \frac{\omega\Big(\hat{b}(x),\hat{b}(p(x))\Big)}{\lambda_{\omega}(K^{\omega})} = \frac{\omega(a^-_{\omega}(\eta(x)),a^+_{\omega}(\eta(x)))}{\lambda_{\omega}(K^{\omega})} = -f^{\omega}_{\mathrm{in}}(\eta(x)),
\end{align*}
where $f^{\omega}_{\mathrm{in}}$ denotes the inner asymmetry function of $\gamma_{\omega}$. Now the one-to-one correspondence between the points of $\partial K$ and $\partial K^{\omega}$ given by $\eta$ yields immediately the equality $\hat{c}_{\mathrm{out}}(\gamma) = c_{\mathrm{in}}(\gamma_{\omega})$.

\end{proof}

\begin{coro} If $(X,\gamma)$ is a strictly convex gauge plane, then its constant of inner asymmetry equals the normalized constant of outer asymmetry of its dual gauge. 
\end{coro}
\begin{proof} A gauge $\gamma$ whose unit disk is $K$ is the dual gauge of the gauge $\gamma_{-K^{\omega}}$, which is isometric to $\gamma_{\omega}$. Hence
\begin{align*} \hat{c}_{\mathrm{out}}(\gamma_{\omega}) =  \hat{c}_{\mathrm{out}}(\gamma_{-K^{\omega}}) = c_{\mathrm{in}}(\gamma),
\end{align*}
and this concludes the proof. This can be thought of as the ``other direction" of the duality.

\end{proof}

As an immediate consequence of the fact that the original measure of outer asymmetry is continuous in the Hausdorff distance (see Theorem \ref{hausdorffouter}), we have that the same holds for the normalized constant of outer asymmetry.

\begin{teo} If $K_n\rightarrow K$ in $\mathcal{K}_{\mathrm{o}}^{\mathrm{sm}}(X)$, then $\hat{c}_{\mathrm{out}}(\gamma_n)\rightarrow \hat{c}_{\mathrm{out}}(\gamma)$. 
\end{teo}

\begin{remark} Notice that, because of the duality given in Theorem \ref{duality}, the convergence in Theorem \ref{hausdorffinner} comes immediately as a corollary of the convergence of the normalized outer measure. As a consequence, Theorem \ref{cont2} is not even necessary to prove the three asymmetry measures introduced here are continuous in the Hausdorff metric. 
\end{remark}

\end{document}